\documentclass[12pt]{amsart}
\usepackage[active]{srcltx}
\usepackage{calc,amssymb,amsthm,amsmath,amscd, eucal,ulem}
\usepackage{alltt}
\RequirePackage[dvipsnames,usenames]{color}

\input{mabliautoref.sty}
\input{xy}
\xyoption{all}
\input{kmacros3.sty}
\usepackage{tikz}
\usepackage{amsfonts, mathrsfs}
\usepackage{calligra}
\usepackage{stmaryrd} 
\usepackage[left=1.1in,top=1in,right=1.1in,bottom=1in]{geometry}

\numberwithin{equation}{theorem}

\newcommand{\ol}{\overline}

\DeclareMathOperator{\length}{length}

\newcommand{\fg}{\pi_1^{\textnormal{\'{e}t}}} 
\DeclareMathOperator{\ns}{ns}
\newcommand{\pSpec}{\text{Spec}^{\circ}} 
\newcommand{\etInCdOne}{\etale{} in codimension $1$}

\usepackage{setspace}
\usepackage{hyperref}
\usepackage{enumerate}
\usepackage{graphicx}
\usepackage[all,cmtip]{xy}

\usepackage{verbatim}

\theoremstyle{theorem}

\renewcommand{\O}{\mathscr O}

\begin{document}
\title{Fundamental groups of $F$-regular singularities via $F$-signature}
\author{Javier Carvajal-Rojas}
\author{Karl Schwede}
\author{Kevin Tucker}
\address{Department of Mathematics\\ University of Utah\\ Salt Lake City\\ UT 84112}
\email{carvajal@math.utah.edu}
\address{Department of Mathematics\\ University of Utah\\ Salt Lake City\\ UT 84112}
\email{schwede@math.utah.edu}
\address{Department of Mathematics\\ University of Illinois at Chicago\\Chicago\\  IL 60607}
\email{kftucker@uic.edu}

\thanks{The first named author was supported in part by the NSF FRG Grant DMS \#1265261/1501115}
\thanks{The second named author was supported in part by the NSF FRG Grant DMS \#1265261/1501115 and NSF CAREER Grant DMS \#1252860/1501102}
\thanks{The third named author was supported in party by NSF Grant DMS \#1419448}


\maketitle

\begin{abstract}
We prove that the local \etale{} fundamental group of a strongly $F$-regular singularity is finite (and likewise for the \'etale fundamental group of the complement of a codimension $\geq 2$ set), analogous to results of Xu and Greb-Kebekus-Peternell for KLT singularities in characteristic zero.  In fact our result is effective, we show that the reciprocal of the $F$-signature of the singularity gives a bound on the size of this fundamental group.  To prove these results and their corollaries, we develop new transformation rules for the $F$-signature under finite \etale-in-codimension-one extensions.  As another consequence of these transformation rules, we also obtain purity of the branch locus over rings with mild singularities (particularly if the $F$-signature is $> 1/2$).  Finally, we generalize our $F$-signature transformation rules to the context of pairs and not-necessarily \etale-in-codimension-one extensions, obtaining an analog of another result of Xu.
\end{abstract}

\section{Introduction}

In \cite[Question 26]{KollarNewExamplesOTandLC} J.~Koll\'ar asked whether if $(0 \in X)$ is the germ of a KLT singularity, then $\pi_1(X \setminus \{ 0\})$ is finite.  In \cite{XuFinitenessOfFundGroups} C.~Xu showed that this holds for the \etale{} local fundamental group, in other words, for the profinite completion of $\pi_1(X \setminus \{0\})$.  Building on this result, \cite{GrebKebekusPeternellEtaleFundamental} proved the finiteness of the \etale{} fundamental groups of the regular locus of KLT singularities (see also \cite{TianXuFinitenessOfFundamentalGroups}).  Over the past few decades, we have learned that KLT singularities are closely related to strongly $F$-regular singularities in characteristic $p > 0$, see \cite{HaraWatanabeFRegFPure,HaraRatImpliesFRat}.  Hence it is natural to ask whether their local \etale{} fundamental groups are also finite.  We show that this is indeed the case.  In fact, we find an upper bound for the size of the fundamental group in terms of a well studied invariant for measuring singularities in characteristic $p > 0$, the $F$-signature $s(R)$.
\begin{theoremA*}[\autoref{thm.FinitenessFundGroup}]
\label{theoremA}
Let $(R,\mathfrak{m}, k)$ be a normal $F$-finite and strongly $F$-regular strictly Henselian\footnote{This just means it is Henselian with separably closed residue field.} local domain of prime characteristic $p>0$, with dimension $d\geq 2$. Then the \'{e}tale fundamental group of the punctured spectrum of $R$, i.e. $\pi_1:=\fg\bigl(\Spec^{\circ}(R)\bigr)$, is finite. Furthermore, the order of $\pi_1$ is at most $1/s(R)$ and is prime to $p$.  The same also holds for $\fg\bigl(\Spec(R) \setminus Z\bigr)$ where $Z \subseteq \Spec R$ has codimension $\geq 2$.
\end{theoremA*}
Observe, that unlike the characteristic zero situation, our characteristic $p > 0$ result is effective, we give an explicit bound on the size of the $\pi_1$.
It is also worth noting that we are working with the \etale{} fundamental group, not the \emph{tame} fundamental group.  Indeed, for $R$ strongly $F$-regular, any finite \etInCdOne{} local extension $(R, \fram, k) \subseteq (S, \fran, \ell)$ must be tame everywhere.  This was already implicitly observed in \cite{SchwedeTuckerTestIdealFiniteMaps} but we make it precise here.  Indeed, we note that $p$ cannot divide $[K(S) : K(R)]$ if the residue fields are equal \autoref{cor.PurityOfBranchLocusForCovers}.

The technical tool where $F$-regularity is used in our proof is a transformation rule for $F$-signature under finite \etale{}-in-codimension-1-morphisms.  The $F$-signature was introduced implicitly in \cite{SmithVanDenBerghSimplicityOfDiff} and explicitly in \cite{HunekeLeuschkeTwoTheoremsAboutMaximal}.  Roughly speaking, it measures how many different ways $R \hookrightarrow F^e_* R$ splits as $e$ goes to infinity.  Explicitly, if $R$ has perfect residue field and $F^e_* R = R^{\oplus a_e} \oplus M$ as an $R$-module, where $M$ has no free $R$-summands, then ${\displaystyle s(R) = {\lim\limits_{e \rightarrow \infty}} {a_e \over p^{e \dim R}}}$.  Here are three quick facts:
\begin{itemize}
\item{}  The limit $s(R)$ exists \cite{TuckerFSigExists}.
\item{}  $s(R) > 0$ if and only if $R$ is strongly $F$-regular \cite{AberbachLeuschke}.
\item{}  $s(R) \leq 1$.
\end{itemize}
Note that there have been a number of transformation rules for $F$-signature under finite maps in the past.  However, they were generally only an inequality (that went the wrong way for our purposes), or assumed that $S$ is flat over $R$ (or made other assumptions about $R$ and $S$).  See for instance  \cite{HunekeLeuschkeTwoTheoremsAboutMaximal,YaoObservationsFSignature,HochsterYaoRationalSignature,TuckerFSigExists}.

\begin{theoremB*}
[\autoref{thm.formula_signature}]
Let $(R,\mathfrak{m}, k) \subseteq (S, \mathfrak{n}, \ell)$ be a module-finite local extension of $F$-finite $d$-dimensional normal local domains in characteristic $p>0$, with corresponding extension of fraction fields $K \subseteq L$.
Suppose $R \subseteq S$ is \etInCdOne, and that $R$ is strongly $F$-regular. Then 
if one writes $S=R^{\oplus f} \oplus M$ as a decomposition of $R$-modules so that $M$ has no nonzero free direct summands, then $f =[\ell: k] \geq 1$ and the following equality holds:
\[
s(S)= \frac{[L:K]}{[\ell:k]} \cdot s(R).
\]
\end{theoremB*}
\noindent
Below, before Theorem D, we discuss how to still get precise transformation rules of $F$-signature even when $R \subseteq S$ is not necessarily \etale{} in codimension 1.

By applying Theorem B in the case $k=\ell$, we see that $s(S) = [L : K] \cdot s(R)$.  Since $s(S) \leq 1$, we immediately see that $[L : K] \leq {1 / s(R)}$.  In other words, the reciprocal of the $F$-signature $s(R)$ gives an upper bound on the generic rank of a finite local \etInCdOne{} extension with the same residue field.  Theorem A then follows.  We also obtain characteristic $p > 0$ corollaries similar to some of those in \cite{GrebKebekusPeternellEtaleFundamental}.

Because our bound on the size of the \etale fundamental group is effective, we immediately obtain a new result on purity of the branch locus.
\begin{theoremC*}[\autoref{cor.PurityOfBranchLocus}]
Suppose $Y \to X$ is a finite dominant map of $F$-finite normal integral schemes.  If $s(\O_{X,P}) > 1/2$ for all $P \in X$ then the branch locus of $Y \to X$ has no irreducible components of codimension $\geq 2$, in other words it is a divisor.
\end{theoremC*}

In \cite{BlickleSchwedeTuckerFSigPairs1}, the notion of $F$-signature of pairs was introduced.  In \autoref{thm.formula_signature_w/Deltas}, we obtain an analogous result to Theorem B in the context of pairs.  Indeed, if $(R, \Delta)$ is a strongly $F$-regular pair, then this can be interpreted as follows.  The reciprocal of $s(R, \Delta)$ gives an upper bound on the generic rank of a finite local extension $(R, \fram) \subseteq (S, \fran)$ such that $\pi^* \Delta - \Ram \geq 0$ (here $\Ram$ is the ramification divisor on $\Spec S$ and $\pi : \Spec S \to \Spec R$ is the induced map).  By taking cones, this immediately yields the following characteristic $p > 0$ analog of the second main result of \cite{XuFinitenessOfFundGroups}.  Here note that globally $F$-regular varieties are an analog of log-Fano varieties in characteristic zero \cite{SchwedeSmithLogFanoVsGloballyFRegular}.

\begin{theoremD*}
[\autoref{cor.CoverOfGloballyFRegPair}]
Suppose that $(X, \Delta)$ is a globally $F$-regular projective pair over an algebraically closed field of characteristic $p > 0$.  There is a number $n$ such that every finite separable cover $\pi : Y \to X$ with $\pi^* \Delta - \Ram \geq 0$ has generic rank $[K(Y) : K(X)] \leq n$.
\end{theoremD*}

\vskip 9pt
\noindent
{\it Acknowledgements:}  The authors would like to thank J\'anos Koll\'ar, Christian Liedtke, Linquan Ma, Lance Miller, Mircea \mustata{}, Stefan Patrikis and David Speyer for valuable and inspiring conversations.  We would like to thank Stefan Kebekus Chenyang Xu and Lance Miller for valuable comments on a previous draft.  We would especially like to thank David Speyer for sharing an early preprint of \cite{SpeyerUnnamed} with us and allowing us to include \autoref{lem.Speyer} which is a special case of his result.

\section{Preliminaries}

\begin{convention}
Throughout this paper, all rings will be assumed to be Noetherian.  They will all be characteristic $p > 0$ unless otherwise stated and they will all be $F$-finite.  All schemes will be assumed to be Noetherian and separated.  If $R$ is an integral domain, then $K(R)$ will denote the fraction field of $R$ (likewise with $K(X)$ if $X$ is an integral scheme).  Given a finite separable map of normal integral schemes $f : Y \to X$,  we use $\Ram$ to denote the ramification divisor on $Y$.
\end{convention}

\subsection{Maps and divisors}

First we fix some notation.  Given a Weil divisor $D$ on $X = \Spec R$, we use $R(D)$ to denote $\Gamma\bigl(X, \O_X(D)\bigr)$.

\begin{definition-proposition}[Maps and Divisors]
Suppose that $R \subseteq S$ is a finite inclusion of normal domains and $A, B$ are divisors on $\Spec R$ and $\Spec S$ respectively with $\pi : \Spec S \to \Spec R$ the canonical map.  Then any nonzero element $\varphi \in \Hom_R\bigl(S(B), R(A)\bigr)$ yields an effective divisor $D_{\varphi} \sim (K_S - B) - \pi^* (K_R - A)$.  If $\varphi,\varphi' \in \Hom_R\bigl(S(B), R(A)\bigr)$ are such that $D_{\varphi} = D_{\varphi'}$ then $\varphi$ and $\varphi'$ are $S$-unit multiples of each other.
\end{definition-proposition}
\begin{proof}
We first notice that
\begin{align*}
\Hom_R\bigl(S(B), R(A)\bigr) \cong & \Hom_R\bigl(S\bigl(B + \pi^* (K_R - A)\bigr), R(K_R)\bigr) \\
 \cong &\Hom_S\bigl(S\bigl(B + \pi^*(K_R - A)\bigr), S(K_S)\bigr) \cong  S\bigl(K_S - B - \pi^*(K_R -A)\bigr).
\end{align*}
Since $\varphi$ is a section of that reflexive sheaf, it yields a divisor $D_{\varphi}$ of zeroes linearly equivalent to $K_S - B - \pi^*(K_R -A)$ as claimed.  Two sections yield the same divisor if and only if they are $S$-unit multiples of each other.
\end{proof}

\begin{lemma}
\label{lem.CompositionOfMapsAndDivisors}
Suppose we have a finite inclusion of normal domains $R \subseteq S \subseteq T$ and we have divisors $A,B,C$ on $\Spec R, \Spec S, \Spec T$ respectively and maps
\[
\beta : T(C) \to S(B), \gamma : S(B) \to R(A).
\]
Then $D_{\gamma \circ \beta} = D_{\beta} + \pi^* D_{\gamma}$ where $\pi : \Spec T \to \Spec R$ is the induced map.
\end{lemma}
\begin{proof}
We work locally and assume that $R$ is a DVR and that $S$ and $T$ are semi-local Dedekind domains and hence are PIDs.  Since $R,S,T$ are all PIDs, $T(C) \cong T, S(B) \cong B, R(A) \cong R$.  By applying these isomorphisms uniformly, we may assume that $C = 0, B = 0, A = 0$.  Let $\Phi \in \Hom_R(S, R)$ be an $S$-module generator, $\Psi \in \Hom_S(T,S)$ be a $T$-module generator, and observe that $\Phi \circ \Psi \in \Hom_R(T, R)$ is a $T$-module generator by, for instance \cite[Lemma 3.9]{SchwedeFAdjunction}.  Write $\gamma(\blank) = \Phi(s \cdot \blank)$, $\beta(\blank) = \Psi(t \cdot \blank)$.  This implies that $D_{\gamma} = \Div_S(s)$ and that $D_{\beta} = \Div_T(t)$.  We observe that $\gamma \circ \beta(\blank) = \Phi \circ \Psi(st \cdot \blank)$ and so \[
D_{\gamma \circ \beta} = \Div_T(st) = \Div_T(t) + \Div_T(s) = D_{\beta} + \pi^* D_{\gamma}
\]
as desired.
\end{proof}

\subsection{$F$-signature}
\label{subsec.FSig}

\begin{definition}[$F$-signature, \cite{HunekeLeuschkeTwoTheoremsAboutMaximal} \cite{BlickleSchwedeTuckerFSigPairs1}]
Suppose that $(R, \fram, k)$ is a $d$-dimensinal $F$-finite local ring with $\alpha = \alpha(R)$ such that $p^{\alpha} = [k : k^p]$.  If we write $F^e_* R = R^{\oplus a_e} \oplus M$ where $M$ has no free $R$-summands, in other words $a_e$ is the maximal rank of a free $R$-summand of $F^e_* R$, then the \emph{$F$-signature of $R$}, is equal to
\[
s(R) := \lim_{e \rightarrow \infty} {a_e \over p^{e(d+\alpha)}}.
\]
More generally, if $\Delta \geq 0$ is a $\bQ$-divisor on $\Spec R$ and if we let $a_e^\Delta$ denote the maximal rank of a free $R$-summand of $F^e_* R$ whose corresponding projection maps lie in $\Hom_R\bigl(F^e_* R(\lceil (p^e - 1)\Delta\rceil), R\bigr) \subseteq \Hom_R(F^e_* R, R)$, then the \emph{$F$-signature of the pair $(R, \Delta)$} is equal to
\[
s(R, \Delta) := \lim_{e \rightarrow \infty} {a_e^\Delta \over p^{e(d+\alpha)}}.
\]
\end{definition}

The elements of $\Hom_R\bigl(F^e_* R(\lceil (p^e - 1)\Delta\rceil), R\bigr) \subseteq \Hom_R(F^e_* R, R)$ form what is called a Cartier algebra since they are closed under composition, \cite{SchwedeTestIdealsInNonQGor,BlickleTestIdealsViaAlgebras}.  Indeed in \cite{BlickleSchwedeTuckerFSigPairs1}, $F$-signature with respect to general Cartier algebras is defined and studied.  Below, we discuss $F$-signature with respect to an object that is not quite a Cartier algebra (but which will otherwise be convenient for us).

Consider $(R, \fram, k)$ an $F$-finite normal local domain with full Cartier algebra $\sC = \sC_{R}$.  If $\Delta \geq 0$ is a $\bQ$-divisor, we can form the Cartier subalgebra $\sC^{\Delta} \subseteq \sC$, with $\sC^{\Delta} = \bigoplus_{e \geq 0} \sC_e^{\Delta} = \bigoplus_{e \geq 0} \Hom_R\bigl(F^e_* R\bigl(\lceil (p^e - 1)\Delta \rceil\bigr), R\bigr)$.  However, it is frequently natural to instead consider $\sG = \sG^{\lfloor p^{\bullet} \Delta \rfloor} =  \bigoplus_{e \geq 0} \sG_e = \bigoplus_{e \geq 0} \Hom_R\bigl(F^e_* R(\lfloor p^e\Delta \rfloor), R\bigr)$.  This object is not generally a Cartier algebra since given two maps
\[
\varphi \in  \Hom_R\bigl(F^e_* R(\lfloor p^e\Delta \rfloor), R\bigr) = \sG_e, \text{ and } \psi \in \Hom_R\bigl(F^f_* R(\lfloor p^f\Delta \rfloor), R\bigr) = \sG_f
\]
we can compose and obtain
\[
\varphi \circ F^e_* \psi : F^{e+f} R\bigl(\lfloor p^f \Delta \rfloor + p^f \lfloor p^e \Delta \rfloor\bigr) \to F^e_* R(\lfloor p^e\Delta \rfloor) \to R.
\]
However, $\lfloor p^f \Delta \rfloor + p^f \lfloor p^e \Delta \rfloor$ is not always $\geq \lfloor p^{e+f} \Delta \rfloor$ and so $\sG$ is not closed under composition.

Observe that if $\Delta \geq 0$ satisfies $\lfloor \Delta \rfloor = 0$, then we have
\begin{equation}
\label{eq.RoundingComparison}
\begin{array}{c}
\lfloor p^e \Delta \rfloor  \leq \lceil (p^e -1) \Delta \rceil \text{ so that}\\
\sG_e = \Hom_R\bigl(F^e_* R(\lfloor p^e\Delta \rfloor), R\bigr) \supseteq \Hom_R\bigl(F^e_* R(\lceil (p^e - 1)\Delta \rceil), R\bigr) = \sC^{\Delta}_e.
\end{array}
\end{equation}
furthermore, we have equality if $(p^e - 1)\Delta$ is an integral Weil divisor.

\begin{setting}
\label{set.SettingForRoundings}
Suppose $(R, \fram)$ is an $F$-finite normal local domain of dimension $d$ and $\Delta \geq 0$ is a $\bQ$-divisor on $\Spec R$ such that $\lfloor \Delta \rfloor = 0$.  We define $\sG_e$ as above.
\end{setting}

\begin{definition}
\label{def.NSandIe}
With notation as in \autoref{set.SettingForRoundings}, we set
\[
\begin{array}{rl}
\sG^{\ns}_e = & \{ \varphi \in \sG_e \;|\; \Image(\varphi) \subseteq \fram \}\\
I_e^{\sG} = & \{ r \in R \;|\; \varphi(F^e_* r) \in \fram, \text{ for all } \varphi \in \sG_e \}.
\end{array}
\]
We also define $a_e^{\sG}$ to be the maximal number of free summands of $F^e_* R$ whose associated projection homomorphisms belong to $\sG_e$.
\end{definition}
We observe that
\begin{equation}
\label{eq.NumberOfSummandsInDifferentGuises}
a_e^{\sG} = \lambda_R\bigl(\sG_e/\sG^{\ns}_e\bigr) = p^{e \alpha} \lambda_R\bigl( R/I_e^{\sG}\bigr)
\end{equation}
by \cite[Lemma 3.6]{BlickleSchwedeTuckerFSigPairs1}, where $\alpha=\alpha(R)=\log_p \bigl([k^{1/p}:k]\bigr)$.

\begin{lemma}\textnormal{[\cf \cite[Lemma 4.17]{BlickleSchwedeTuckerFSigPairs1}]}
\label{lem.FSigLimitWithGs}
With notation as in \autoref{set.SettingForRoundings}
\[
\lim_{e \rightarrow \infty} { a_e^{\sG} \over p^{e (d+\alpha)}} = s(R, \Delta).
\]
Moreover, the same result holds even if we replace $\sG_e$ by the asymptotically small perturbation of it:
\[
\sG'_e:= \Hom_R\bigl(F^e_* R\bigl(\lfloor p^e\Delta \rfloor + D \bigr), R\bigr)
\]
for any other effective Weil divisor $D$.
\end{lemma}

\begin{proof}
Choose $0 \neq c \in R$ such that
\[
(F^e_* c) \cdot \sC^{\Delta}_e \subseteq (F^e_* c) \cdot \sG_e \subseteq \sC^{\Delta}_e \subseteq \sG_e.
\]
Indeed, any $c$ with $\Div(c) \geq \lceil \Delta \rceil$ will work.  Obviously we have that
\[
\bigr( I_e^{\sG} :_R c \bigl)= \{ r \in R \; |\; \varphi(F^e_* r) \in \fram, \text{ for all } \varphi \in (F^e_* c) \cdot \sG_e \}.
\]
Thus, the above inclusions $(F^e_* c) \cdot \sG_e \subseteq \sC^{\Delta}_e \subseteq \sG_e$ induce inclusions of ideals:
$$I_e^{\sG} \subseteq I_e^{\Delta} \subseteq \bigl(I_e^{\sG}:_R c\bigr) .$$
Now simply apply \autoref{lem.CTrick} below.

For the last assertion in the statement of the lemma, we just choose a $0 \neq c \in R$ such that $\Div(c) \geq \Delta + D$, then we are going to have $(F^e_* c) \cdot \sG_e \subseteq \sG'_e \subseteq \sG_e$, thus we argue as above, including running \autoref{lem.CTrick}.
\end{proof}

\begin{lemma}
\label{lem.CTrick}
Let $(R,\fram, k)$ be an $F$-finite normal local domain of dimension $d$. Suppose we have a pair of sequences of finite colength ideals $\{I_e\}_e$ and $\{J_e\}_e$ such that $\fram^{[p^e]} \subseteq I_e \subseteq J_e$ holds for all $e$. Additionally assume that there is a $0 \neq c \in R$ such that $I_e \subseteq J_e \subseteq (I_e:c)$ for all $e$. Then, it follows that:
\[
\lim_{e \rightarrow \infty}{\frac{1}{p^{de}}}\lambda_R(J_e/I_e) = 0.
\]
In particular,
\[
\lim_{e \rightarrow \infty}{\frac{1}{p^{de}}}\lambda_R(R/J_e) = \lim_{e \rightarrow \infty}{\frac{1}{p^{de}}}\lambda_R(R/I_e)
\]
provided that any (then both) of the limits exists.
\end{lemma}
\begin{proof}
Consider the four term exact sequence
\[
0 \to (I_e:c)/I_e \to R/I_e \overset{\cdot c}{\longrightarrow} R/I_e \to R/(I_e+cR) \to 0.
\]
From this we conclude \cf \cite[Page 8]{VraciuNewVersion},
$$ \lambda_R(J_e/I_e) \leq \lambda_R\bigl((I_e:c)/I_e\bigr) = \lambda_R\bigl(R/(I_e+cR)\bigr) $$
Since $\fram^{[p^e]} \subseteq I_e$ there is a constant $C$ such that $\lambda_R\bigl((R/(I_e + cR)\bigr) \leq C p^{e(d-1)}$. The result then follows.
\end{proof}

We conclude this subsection with a brief recollection of strongly $F$-regular singularities and globally $F$-regular projective varieties.

\begin{definition}
Suppose that $(R, \fram)$ is a normal local ring and $\Delta \geq 0$ is a $\bQ$-divisor.  We say that $(R, \Delta)$ is \emph{strongly $F$-regular} if $s(R, \Delta) > 0$.  This is equivalent by \cite{AberbachLeuschke,BlickleSchwedeTuckerFSigPairs1} to the assertion that for every $0 \neq c \in R$, there exists some $e > 0$ and $\varphi \in \sC^{\Delta}_e$ with $\varphi(F^e_* c R) = R$.

Suppose $X$ is a normal projective variety over an algebraically closed field and $\Delta \geq 0$ is a $\bQ$-divisor on $X$.  For any ample line bundle $\sL$ on $X$, form the section ring $R = \bigoplus_{i \geq 0} H^0(X, \sL^i)$ and let $\Delta_R$ be the corresponding $\bQ$-divisor.  We say that $(X, \Delta)$ is \emph{globally $F$-regular} if $(R, \Delta_R)$ is strongly $F$-regular, this is independent of the choice of $\sL$ \cite{SchwedeSmithLogFanoVsGloballyFRegular}.
\end{definition}

\subsection{Notes on trace}

We believe the following easy lemma is well known to experts but we do not know an easy reference

\begin{lemma}
\label{lem.Note_Trace}
Suppose that $(R, \fram) \subseteq (S, \fran)$ is a finite extension of normal local rings. Then $\Tr(\fran) \subseteq \fram$.
\end{lemma}
\begin{proof}
Choose $A$ a divisorial discrete valuation ring with uniformizer $a \in A$ of $K(R)$ centered over $V(\fram) \subseteq \Spec R$.  Let $B$ be the normalization of $A$ inside $K(S)$.  Note that $B$ is a 1-dimensional semi-local ring with maximal ideals $\frab_1, \ldots, \frab_l$.  Note that each $\frab_i \cap S = \fran$.  Let $K_{B/A}$ be the relative canonical divisor/ramification divisor of $A \subseteq B$ and observe that\footnote{This follows from the assertion that the section $\Tr \in \Hom_A(B,A)$ corresponds to the ramification divisor, see \cite{SchwedeTuckerTestIdealFiniteMaps}.} $\Tr B(K_{B/A}) \subseteq A$.  Now, notice that
\[
\Tr\bigl( a \cdot B(K_{B/A})\bigr) \subseteq \langle a \rangle
\]
by linearity of trace.  We claim that $\bigcap_i \frab_i \subseteq a \cdot B(K_{B/A})$.  We can check this locally on $B$.  Indeed, localize at a $\frab_i$ to obtain a DVR $B' = B_{\frab_i}$ with uniformizer $b$.  Then write $a = ub^n$ where $u$ is a unit of $B'$.  We also know that $K_{B'/A} = m\Div(b)$ where $n-1 \leq m$ (with equality in the case of tame ramification).  $B'(K_{B'/A}) = {1 \over b^m} B'$ and so \[
a \cdot B'(K_{B'/A}) = {a \over b^m} B' = {b^{n-m} B'} \supseteq b B'
\]
where the last containment holds since $n-m \leq 1$.  In conclusion
\[
\Tr\biggl(\bigcap_i \frab_i\biggr) \subseteq \Tr\bigl( a \cdot B(K_{B/A})\bigr) \subseteq \langle a \rangle.
\]

Next choose $x \in \fran$, then $x \in \bigcap_i \frab_i$ and so $\Tr(x) \in aA \cap R = \fram$ as claimed.
\end{proof}

The following corollary can also be viewed as a result on purity of the branch locus.  We also believe it is well known to experts but we do not know a reference.

\begin{corollary}
\label{cor.PurityOfBranchLocusForCovers}
Suppose that $(R,\fram)$ is a splinter (for instance, if it is strongly $F$-regular).  Then there is no finite local extension $(R,\fram) \subseteq (S,\fran)$ of normal domains with corresponding fraction fields $K \subseteq L$ such that
\begin{enumerate}
\item the residue fields $R/\fram = S/\fran = k$ are equal,
\label{cor.PurityOfBranchLocusForCovers.1}
\item $p \;\big|\; [L : K]$, and
\label{cor.PurityOfBranchLocusForCovers.2}
\item $R \subseteq S$ is \etInCdOne.
\label{cor.PurityOfBranchLocusForCovers.3}
\end{enumerate}
\end{corollary}
\begin{proof}
Suppose that there is such an extension.
By \autoref{cor.PurityOfBranchLocusForCovers.3}, we know that $\Tr : S \to R$ is surjective, see \cite[Proposition 7.4]{SchwedeTuckerTestIdealFiniteMaps}.  By \autoref{lem.Note_Trace} $\Tr(\fran) \subseteq \fram$ and so we have an induced surjective map
\[
\overline{\Tr} : S/\fran \to R/\fram.
\]
But this is a $k$-linear map between $1$-dimensional $k$-vector spaces by \autoref{cor.PurityOfBranchLocusForCovers.1}, and hence an isomorphism.  Consider $1 \in R \subseteq S$.  We know $\Tr(1) = [L : K] \cdot 1 = 0$ by \autoref{cor.PurityOfBranchLocusForCovers.2}.  Hence $\overline{\Tr}(1) = 0$ which is a contradiction.
\end{proof}

Note in the above proof, if the condition that $R$ is a splinter is removed, then condition \eqref{cor.PurityOfBranchLocusForCovers.3} can be replaced by the condition that $\Tr : S \to R$ is surjective, and then the result is well known to experts (with the same proof).

\begin{definition}
Let $R\subseteq S$ be a finite separable extension of normal domains, we say that a direct $R$-summand $M$ of $S$ is a \emph{$\Tr$-summand} if $M\cong R$ and the associated projection linear map $\rho: S \to M \overset{\cong}{\longrightarrow} R$ is such that $D_{\rho}\geq \Ram$, in other words, $\rho=\Tr(s\cdot \blank)$ for some $s\in S$, equivalently $\rho \in \Tr\cdot S$ inside the $S$-module $\Hom_S(S,R)$. This is independent of the choice of the isomorphism $M \cong R$.
\end{definition}

We will use the next lemma frequently.
\begin{lemma}
\label{lem.TrSurjectsImpliesTrSummandsIslk}
Suppose $(R, \fram, k) \subseteq (S, \fran, \ell)$ is a separable finite extension of normal local domains.  If $\Tr : S \to R$ is surjective, then $[\ell : k]$ is equal to the number of simultaneous free $R$-summands of $S$ whose projection maps are multiples of $\Tr$, in other words the number of $\Tr$-summands.  In particular if $K = K(R)$, $L = K(S)$ and $[\ell : k] = [L : K]$, then $S$ is a free $R$-module.
\end{lemma}
\begin{proof}
Consider $D = \Tr \cdot S \subseteq \Hom_R(S, R)$.  Notice that $\langle \varphi \in D\;|\; \varphi(S) \subseteq \fram \rangle = \Tr \cdot S$ by \autoref{lem.Note_Trace}.
By \cite[Lemma 3.6]{BlickleSchwedeTuckerFSigPairs1}, the number of free summands whose projection maps are multiples of $\Tr$ is equal to
\[
\length_R\left({\Tr \cdot S \over \Tr \cdot \fran}\right) = \length_R(\ell) = [\ell : k]
\]
which proves the first statement.  For the second, if we have equality, then we have a surjective map $S \to R^{\oplus [\ell : k]}$.  But $S$ is torsion free of generic $R$-rank $[L : K] = [\ell : k]$.
\end{proof}

Let us now briefly discuss tame ramification. In the case that $K(R) \subseteq K(S)$ is Galois, the surjectivity of the trace map $\Tr: S \to R$ has been coined \emph{cohomologically tamely ramified} by M. Kerz and A. Schmidt; see for example \cite[Claim 1, Theorem 6.2]{KerzSchmidtOnDifferentNotionsOfTameness}. They actually proved in that article that this is in fact the strongest among all the notion of tameness, including the one in \cite{GrothendieckMurreTameFundamentalGroup}.

What we want to observe next is that in the cases we consider in this paper, if $\Tr:S \to R$ is surjective, then the extension of residue fields $k \subseteq \ell$ is a separable extension.  This is implicit in the proof of \cite[Claim 1, Theorem 6.2]{KerzSchmidtOnDifferentNotionsOfTameness} but we give a careful proof in our setting.  First we recall a very special case of a result of David Speyer.  We thank David Speyer for sharing a preliminary draft of a paper with us.

\begin{lemma}[\cite{SpeyerUnnamed}]
\label{lem.Speyer}
Suppose that $(R, \fram) \subseteq (S, \fran)$ is a finite local extension of $F$-finite normal domains with $\Tr : S \to R$ surjective.  Further suppose that $\varphi_R : F^e_* R \to R$ is an $R$-linear map that extends to $\varphi_S : F^e_* S \to S$.  If $\fram$ is $\varphi_R$-compatible, then $\fran$ is $\varphi_S$-compatible.
\end{lemma}
\begin{proof}
We give a short proof here for the convenience of the reader.
We have the commutative diagram \cite[Proposition 4.1]{SchwedeTuckerTestIdealFiniteMaps}
\[
\xymatrix{
F^e_* S \ar[d]_{F^e_* \Tr} \ar[r]^-{\varphi_S} & S \ar[d]^{\Tr} \\
F^e_* R \ar[r]_-{\varphi_R} & R.
}
\]
We know that $\varphi_R\bigl(F^e_* \Tr (\fran)\bigr) \subseteq \varphi_R(F^e_* \fram) \subseteq \fram$.  Hence, we also have that
\[
\Tr\bigl(\varphi_S (F^e_* \fran)\bigr) \subseteq \fram.
\]
Thus $\varphi_S (F^e_*  \fran) \subsetneq S$, since $\Tr$ is onto, and hence $\varphi_S(F^e_* \fran) \subseteq \fran$ as desired.
\end{proof}

\begin{lemma}
\label{lem.SepResFields}
Suppose $(R, \fram, k) \subseteq (S, \fran, \ell)$ is a finite local extension of $F$-finite normal domains that is \etale{} in codimension $1$ and such that $R$ is strongly $F$-regular.  Let $K = K(R)$ and $L = K(S)$.  In this case $k \subseteq \ell$ is separable and $[\ell : k] \;\big|\; [L : K]$.
\end{lemma}
\begin{proof}
We begin with a claim.
\begin{claim}
There exists a surjective map $\varphi_R : F^e_* R \to R$ such that $\varphi_R (F^e_* \fram) \subseteq \fram$.
\end{claim}
\begin{proof}[Proof of claim]
Choose a surjective $\psi : F^e_* R \to R$ (which exists since $R$ is $F$-split). Then there exists a smallest $j > 0$ such that $\psi(F^e_* \fram^j) \subseteq \fram$. If $j = 1$, we may take $\varphi = \psi$, so assume $j > 1$.  By hypothesis $\varphi(F^e_* \fram^{j-1}) = R$. Choose $z \in \fram^{j-1}$ with $\varphi(F^e_* z) = 1$.  Form the map $\varphi_R(F^e_* \blank) = \psi\bigl(F^e_* (z \cdot \blank)\bigr)$, it has the desired properties.  This proves the claim.
\end{proof}

Since $R \subseteq S$ is \etale{} in codimension 1, $\varphi_R$ extends to a map $\varphi_S$ by \cite[Theorem 5.7]{SchwedeTuckerTestIdealFiniteMaps}.
Therefore we have the commutative diagram where the vertical maps are inclusions.
\[
\xymatrix{
F^e_* R \ar@{^{(}->}[d] \ar[r]^-{\varphi_R} & R \ar@{^{(}->}[d] \\
F^e_* S  \ar[r]_-{\varphi_S} & S.
}
\]
We notice that $\fran$ is $\varphi_S$-compatible by \autoref{lem.Speyer}.
Thus by modding out the bottom row by the compatible ideal $\fran$ and the top row by $\fram = \fran \cap R$ we obtain the diagram
\[
\xymatrix{
F^e_* k \ar@{^{(}->}[d] \ar[r]^-{\varphi_k} & k \ar@{^{(}->}[d] \\
F^e_* \ell  \ar[r]_-{\varphi_{\ell}} & \ell.
}
\]
It follows from \cite[Proposition 5.2]{SchwedeTuckerTestIdealFiniteMaps} that $k \subseteq \ell$ is separable since $\varphi_k$ is surjective and hence not the zero map.  This proves the first result.

After completion, the second part follows from the standard claim below.
\begin{claim}
Suppose $(R, \fram, k) \subseteq (S, \fran, \ell)$ is a finite local inclusion of complete $F$-finite normal local rings with $k \subseteq \ell$ separable.  Let $K = K(R)$ and $L = K(S)$.  Then
\[
[\ell : k] \;\big|\; [L : K].
\]
\end{claim}
\begin{proof}[Proof of claim]
Since the extension of coefficient fields is separable, we can choose $k \subseteq \ell$ a containment of coefficient fields as well.
Now consider $R \otimes_k \ell$.  Obviously this maps to $S$.  Hence we have a factorization
\[
R \to R \otimes_k \ell \to S.
\]
Now $R \otimes_k \ell$ is a finite extension of $R$ with unique maximal ideal $\fram \otimes_k \ell$.  Also, since $\ell /k$ is geometrically normal, $R \otimes_k \ell$ is normal by \cite[Tag 06DF]{stacks-project}.  It follows that $R \otimes_k \ell$ is a normal local domain and a finite extension of $R$ with a map to $S$.  The map $R \otimes_k \ell \to S$ is also finite and must be injective because if it had a kernel, $S$ would be finite over a lower dimensional ring.  Hence $R \otimes_k \ell \subseteq S$. Let $T = K(R \otimes_k \ell)$.  Then we have $K \subseteq T \subseteq L$ and $[T : K] = [\ell : k]$ (since $R \otimes_k \ell$ is a free $R$-module of rank $[\ell : k]$).
This finishes the claim.
\end{proof}
It also finishes the proof of \autoref{lem.SepResFields}.
\end{proof}

\begin{remark}
It is not difficult to generalize \autoref{lem.SepResFields} to the context of pairs if $R \subseteq S$ is not necessarily \etInCdOne{}.  The only places where \etInCdOne{} is used is when we extend $\varphi_R$ to $\varphi_S$.  Indeed, if one assumes that $(R, \Delta)$ is strongly $F$-regular and that $\pi^* \Delta - \Ram$ is effective, the proof works without change.  We will not need this generalization in what follows, however.
\end{remark}

\subsection{Local fundamental groups of singularities}

The study of local fundamental groups of (normal) singularities has a long history, having early origins in the study of resolution of singularities in positive characteristic based on the work of S. Abhyankar and others (this mostly related to fundamental groups of a curve singularities). It also goes back for example to the work of D. Mumford \cite{MumfordFundGroup} in which it is proven that for the analytic germ of a normal complex surface, regularity or smoothness is equivalent to the triviality of the local fundamental group.  
The same principle was generalized to the algebraic setting by H. Flenner in \cite{FlennerReineLokaleRinge}, see also \cite[Corollary 5]{CutkoskySrinivasanFundGroupSurfacePositiveChar}. This however is known to be false in positive characteristic by examples given by M. Artin in \cite{ArtinCoveringsOfTheRtionalDoublePointsInCharacteristicp}, see also \cite{CutkoskySrinivasanFundGroupSurfacePositiveChar}. Our results are focused on higher dimensions however.

By local fundamental group of a singularity here we understand and mean the \etale{} fundamental group of the punctured spectrum of a strictly local\footnote{Also called strictly Henselian, which means Henselian with separably closed residue field.} normal domain $(R, \fram, k)$, \ie $\fg\bigl(\pSpec(R)\bigr)$, as defined in \cite[Expos\'{e} V]{GrothendieckSGA} or well in \cite{MurreLecturesFundamentalGroups}. Notice that since $R$ is a domain, its punctured spectrum is connected, thereby it makes sense to talk about its \'{e}tale fundamental group.

As customary for normal schemes, we choose the generic point as our base point $\bar{x}$, \ie our base point is going to be the field extension $\bar{x} : K \hookrightarrow K^{\text{sep}}$, the or some separable closure of $K$ the fraction field of $R$.

We will need to compare \etale{} coverings of a normal connected scheme $X$ and the \etale{} coverings of an open connected subscheme $U$ of it, say $U=X\setminus Z$ with $Z \subseteq X$ of codimension $\geq 2$. Then we need to observe the following two Galois categories are equivalent \cf \cite{ArtinCoveringsOfTheRtionalDoublePointsInCharacteristicp} for the local case we care about: the category  $\mathscr{G}$ of normal schemes finite over $X$ which are \'{e}tale except possibly above $Z$; and the category $\mathsf{FEt}/U$ of \'{e}tale coverings of $U$. We clearly have a functor $\mathscr{G} \to \mathsf{FEt}/U$ given by restriction to $U$, i.e. $(f:Y \to X) \longmapsto (f: f^{-1}(U) \longmapsto U)$, or as more classically described in this context $Y/X \longmapsto (Y\times_X U) / U$. In fact, this functor gives rise to an equivalence of Galois categories. This is nothing but a consequence of Zariski's Main Theorem, see for instance \cite[Theorem 1.8, Chapter 1]{MilneEtaleCohomology}. For if we have an \'{e}tale covering $V \to U$, then the composition or extension $V \to U \subseteq X$ is a quasi-finite morphism, hence there exists a scheme $Y$ together with a finite morphism  $Y \to X$ such that $V$ can be realized as an open subscheme of $Y$ and the quasi-finite morphism above $V \to X$ factors as $V \subseteq Y \to X$. In other words, any \'{e}tale covering $V\to U$ can be realized as the restriction of a finite map $Y \to X$ which is \'{e}tale everywhere except possibly above the complement of $U$, that is \etale{} on $V$. Let us use the terminology, any \'etale cover of $U$ extends to a cover of $X$.

The above observation will be of great help for us inasmuch as by definition $\pi_1(U)$ is the fundamental group classifying or pro-representing the Galois category $\mathsf{FEt}/U$, but we use the equivalence above to classify instead the category $\mathscr{G}$. In case $X=\Spec(R)$ with $R$ a singularity as above, this category is the same as the category of module-finite (semi-local) inclusions of normal rings $R\subseteq S$ (where the morphism are the morphism of $R$-algebras) which are \'{e}tale except possibly at the prime ideals lying over the closed $Z$ (or $\fram$ when $U$ is the punctured spectrum of $R$). Hence, we can work in a ring-theoretic setting.

Another advantage is obtained in the understanding of the Galois coverings. For remember that in order to compute $\pi_1$ one takes the directed projective system of connected/minimal and Galois elements $R \subseteq S$ in $\mathscr{G}$, and then considers its associated projective limit
$$ \pi_1 = \varprojlim{\Aut_{\mathscr{G}}{S}}. $$
However, since the base normal domain $R$ is Henselian, the connected/minimal elements of $\mathscr{G}$ are precisely the local domains in $\mathscr{G}$, this is precisely the equivalence between (a) and (b) in \cite[Theorem 4.2, Chapter 1]{MilneEtaleCohomology}. The Galois extensions here are those $R \subseteq S$ with finite Galois extensions of fraction fields, moreover in that case $\Aut_{\mathscr{G}}{S}=\Hom_R(S, K^{\mathrm{sep}})=\mathrm{Gal}(L/K)$ where $L$ is the fraction field of $S$.

We finish this series of remarks by pointing out that for general normal schemes there is a particularly nice way to describe the computation of the respective \etale{} fundamental group. Say $S$ is a normal and connected scheme, consequently integral, so if we choose the base point to be a fixed separable closure of the fraction field $K$, we can take $X_i$ to be the normalization of $X$ in $K_i$ where the $K_i$ are the finite Galois extension of $K$ inside of our fixed separable closure such that $X_i$ is unramified over $X$, then
$$\pi_1 = \varprojlim \mathrm{Gal}(K_i/K) .$$
\section{$F$-Signature goes up under the presence of ramification}

Suppose $(R,\mathfrak{m}, k) \subseteq (S, \mathfrak{n}, \ell)$ is a module-finite local extension of $F$-finite characteristic $p>0$ $d$-dimensional local domains with corresponding extension of fraction fields $K \subseteq L$. According to \cite[Corollary 4.13]{TuckerFSigExists} if one writes $S=R^{\oplus f} \oplus M$ as a decomposition of $R$-modules so that $M$ has no nonzero free direct summands, then the following inequality relating the $F$-signatures holds:
$$ f \cdot s(S) \leq [L:K] \cdot s(R). $$
In the theorem below, we show that if the extension above is \etInCdOne{} and if $R$ is strongly $F$-regular, then equality holds.  
We will also observe that if $\ell = k$, then $f = 1$. So that if there is any ramification (i.e. the extension is not \'{e}tale everywhere) then the $F$-signature of $S$ would be at least twice the $F$-signature of $R$.
\begin{theorem}\label{thm.formula_signature}
Let $(R,\mathfrak{m}, k) \subseteq (S, \mathfrak{n}, \ell)$ be a module-finite local extension of $F$-finite $d$-dimensional normal local domains in characteristic $p>0$, with corresponding extension of fraction fields $K \subseteq L$.
Suppose $R \subseteq S$ is \etInCdOne, and that $R$ is strongly $F$-regular. Then 
if one writes $S=R^{\oplus f} \oplus M$ as a decomposition of $R$-modules so that $M$ has no nonzero free direct summands, then $f =[\ell: k] \geq 1$ and the following equality holds:
\[
s(S)= \frac{[L:K]}{[\ell:k]} \cdot s(R).
\]
\end{theorem}
\begin{proof}
We notice that if $R$ is strongly $F$-regular, so is $S$ by \cite[Theorem 2.7]{WatanabeFRegularFPure}, so both $R$ and $S$ have positive $F$-signature by \cite{AberbachLeuschke}.  Also see \cite[Lemma 3.5, Lemma 3.6]{SchwedeTuckerTestIdealFiniteMaps}. 

By \cite[Proposition 4.8]{SchwedeTuckerTestIdealFiniteMaps}, it follows that the trace map $\Tr: S \to R$ generates the $S$-module $\Hom_R(S,R)$. Moreover, by \cite[Corollary 7.7]{SchwedeTuckerTestIdealFiniteMaps}, the trace map is surjective (this following from $R$ being strongly $F$-regular).  By \autoref{lem.TrSurjectsImpliesTrSummandsIslk} $f = [\ell:k]$.

Having the above in mind, we follow the proof of \cite[Corollary 4.13]{TuckerFSigExists}, trying to improve the estimates there by using these stronger conditions.
We notice that
\[
\alpha(R):=\log_p\bigl([k^{1/p}:k]\bigr)=\log_p\bigl([\ell^{1/p}:\ell]\bigr)=:\alpha(S),
 \]
and denote by $b_e$ the maximal rank of a free $R$-module appearing in a direct sum decomposition of $S^{1/p^e}$. If one writes a decomposition $S^{1/p^e}=S^{\oplus a_e(S)} \oplus N_e$ as $S$-modules where $N_e$ does not admit a free direct summand as $S$-module, then one also gets a decomposition of $S^{1/p^e}$ as an $R$-module
\[
S^{1/p^e} = \bigl(R^{\oplus f} \oplus M\bigr)^{a_e(S)} \oplus N_e = R^{f \cdot a_e(S)} \oplus M^{\oplus a_e(S)} \oplus N_e.
\]
From this one concludes that $b_e \geq f \cdot a_e(S)$.  However, equality might not hold because of the possibility of free $R$-summands coming from $N_e$.  We will show that this cannot happen under our stronger hypotheses. In fact, if $N$ is any $S$-module with no nonzero free direct $S$-summands, we will show it has no free direct $R$-summands as well. Indeed, by \cite[Lemma 3.9]{SchwedeFAdjunction} one has that any $R$-linear map $N \to R$ is going to admit a factorization through the trace map, \ie $N\to S \xrightarrow{\Tr} R$ (for this we are making use of the condition $\Hom_R(S,R)=\Tr \cdot S$). Then, by virtue of the inclusion $\Tr(\mathfrak{n}) \subseteq \mathfrak{m}$ \autoref{lem.Note_Trace}, if $N \to R$ is surjective so has to be the factor $N\to S $, so that any free $R$-summand of $N$ would give rise to a free $S$-summand.

In conclusion, we have that $b_e = f \cdot a_e(S)$. By dividing through  by $p^{e(d+\alpha(R))}= p^{e(d+\alpha(S))}$ and letting $e \to \infty$ one obtains the desired equality (this by making use of \cite[Theorem 4.11]{TuckerFSigExists}).
\end{proof}

The following corollary will be one of the key ingredients to show the finiteness of the \'{e}tale fundamental group of a strongly $F$-regular singularity. It reflects how the $F$-signature imposes strong conditions for the existence of non trivial coverings of a punctured spectrum.

\begin{corollary}
\label{cor.F_Signature_goes_up_under_ramification}
With the same setup as in \autoref{thm.formula_signature}, if the inclusion is not \'{e}tale everywhere, then $s(S)\geq 2s(R)$.
\end{corollary}
\begin{proof}
Since in our case, $[\ell : k] \;|\; [L : K]$ by \autoref{lem.SepResFields}, it is sufficient to show that $[L : K] > [\ell : k]$.  But if we have equality, then $S$ is a free $R$-module by \autoref{lem.TrSurjectsImpliesTrSummandsIslk} and hence $R \subseteq S$ is \etale{} everywhere by \cite[Chapter VI, Theorem 6.8]{AltmanKleimanIntroToGrothendieckDuality}.
%
\end{proof}


Now we state and prove our promised purity of the branch locus result.  Compare with \cite{ZariskiPurity,NagataPurity,KunzRemarkOnPurityOfBranchLocus, SGA2, CutkoskyPurity}.

\begin{corollary}[Purity of the branch locus for rings with mild singularities]
\label{cor.PurityOfBranchLocus}
Suppose $Y \to X$ is a finite dominant map of $F$-finite normal integral schemes.  If $s(\O_{X,P}) > 1/2$ for all $P \in X$ then the branch locus of $Y \to X$ has no irreducible components of codimension $\geq 2$, in other words it is a divisor.
\end{corollary}
\begin{proof}
We work locally with $X = \Spec R$ and $Y = \Spec S$.
Let $P \subseteq R$ be a minimal prime of the locus where $R \subseteq S$ is not \'{e}tale, \ie a branch point. Suppose however the height of $P$ is at least $2$. Localizing at $P$ and completing, we may form $\bigl(\widehat{R_P}, \fram, k\bigr) \subseteq \prod S_i$ where the $S_i$ are complete normal local domains.  Since the original $R_P \subseteq S_P$ is not \etale, at least one of the finite inclusions $\widehat{R_P} \subseteq S_i$ is not \etale.  Switching notation, let $R = \widehat{R_P}$ and $S = S_i$ so that we have a finite inclusion $(R, \fram, k) \subseteq (S, \fran, \ell)$ with $s(R) > 1/2$.  We notice that this extension is \etInCdOne{} but not \etale.  But this contradicts \autoref{cor.F_Signature_goes_up_under_ramification} and the fact that $s(S) \leq 1$.
\end{proof}

\begin{remark}
Kunz began using Frobenius to study singularities because he observed that if $F^e_* R$ is a free $R$-module, then purity of the branch locus holds for $R$.  On the other hand, $s(R) > 1/2$ can be interpreted as saying that $F^e_* R$ is more than half-free as an $R$-module (at least for $e \gg 0$).  In other words, the above says that if $F^e_* R$ is more than half-free as an $R$-module for $e \gg 0$, then purity of the branch locus holds.
\end{remark}


We rephrase this in one more way, and point out that we can do slightly better in characteristic 2.


\begin{corollary}  
\label{cor.BoundGenericRnk}
Suppose $(R, \fram, k)$ is a local $F$-finite strongly $F$-regular domain with fraction field $K = K(R)$.  Then $\lfloor {1 / s(R)}\rfloor$ is an upper bound on the size of $[L : K]/[\ell : k]$, where $L = K(S)$ and $(R, \fram, k) \subseteq (S, \fran, \ell)$ is an extension of normal local domains that is \etInCdOne.  In the case that $k$ is separably closed, $\lfloor {1 / s(R)}\rfloor$ is an upper bound on the maximum size of $[L : K]$ over any such \etale-in-codimension-one extension.  Returning to the case of a general choice of $k$, if
\begin{itemize}
\item[(i)] $s(R) > 1/2$ or,
\item[(ii)]  $s(R) > 1/3$, $p = 2$ and $k$ is separably closed,
\end{itemize}
then there is no proper finite local extension of normal domains $(R, \fram) \subseteq (S, \fran)$ that is \etale{} in codimension 1 but not \etale{} (N.B. this implies triviality of the respective local fundamental group).
\end{corollary}
\begin{proof}
Suppose that $(R, \fram, k) \subseteq (S, \fran, \ell)$ is a finite local extension that is \etale{} in codimension 1 and with $g := [L : K]/[\ell : k] > \lfloor {1 / s(R)}\rfloor$.  Hence by \autoref{thm.formula_signature}, we obtain $s(S) = g \cdot s(R) > 1$ a contradiction.  This proves the first statement.  Next if $k$ is separably closed, then $k= \ell$ by \autoref{lem.SepResFields} which proves the second.  Part (i) is just a local restatement of \autoref{cor.PurityOfBranchLocus}.   For (ii), \autoref{cor.PurityOfBranchLocusForCovers} implies that $g > 2$.  The result follows.
\end{proof}




\begin{example} It is known that the $F$-signature of $\ol{\mathbb{F}}_p \llbracket x_0,...,x_3 \rrbracket/(x_0^2+\cdots + x_3^2)$ is $\mbox{3/4}$, see \cite[Theorem 3.1]{WatanabeYoshidaHKMultiplicityOfThreeDimensionalRings}, \cite[Proof of Theorem 11]{HunekeLeuschkeTwoTheoremsAboutMaximal}, \cite[Proof of Proposition 4.22]{TuckerFSigExists}.  Of course, it is well known that this complete intersection satisfies purity of the branch locus \cite{SGA2}.  
\end{example}

\begin{corollary}
Suppose that $\sL$ is an ample line bundle on a globally $F$-regular projective variety $X$ over an algebraically closed field $k$.  Write $R = \bigoplus_{i \geq 0} H^0(X, \sL^i)$.  If $\sL = \sA^m$ for another line bundle $\sA$, with $p \not|\;\, m$, then $m \leq {1 / s(R)}$.
\end{corollary}
\begin{proof}
We have an inclusion $R \subseteq S = \bigoplus_{i \geq 0} H^0(X, \sA^i)$ which, while not graded, simply multiplies degrees by $m$.  We notice that this inclusion is \etInCdOne, see for instance \cite[Lemma 5.7]{SchwedeSmithLogFanoVsGloballyFRegular}.  Furthermore, it is easy to see that it has generic rank $m$.  The result follows immediately by completion (or Henselization) and \autoref{thm.formula_signature}.
\end{proof}

\section{Behavior of $F$-signature of pairs under finite morphisms}

In this section we generalize the formula of \autoref{thm.formula_signature} to maps which are not necessarily \etInCdOne.  We include it in a separate section because of its much more technical proof.
We will see that in compensation, the formalism of divisors provides a way to include the presence of codimension-1 ramification into the transformation formula, so that we can recover the formula without pairs in the case of extensions that are \etInCdOne.  

As in the proof of \autoref{thm.formula_signature}, one of the main ingredients will be the capability of factoring through the trace maps. The following two lemmas are formulated here as generalization of the sort of factorization in \cite[Lemma 3.9]{SchwedeFAdjunction}, we do it in two steps; we first add pure ramification with \autoref{lem.TraceComposedWithRamMaps} and secondly the presence of pairs with \autoref{lem.TraceComposedWithRamMapsAndDeltas}.

\begin{lemma}
\label{lem.TraceComposedWithRamMaps}
Suppose that $R \subseteq S$ is a finite separable extension of normal domains with ramification divisor $\Ram$ on $\Spec S$. Then the image of the map induced by composition
\[
\Hom_S\bigl(F^e_* S(\Ram), S(\Ram)\bigr) \times \Hom_R\bigl(S(\Ram), R\bigr) \to \Hom_R(F^e_* S, R)
\]
contains $\Hom_R\bigl(F^e_* S(\Ram), R\bigr) \subseteq \Hom_R(F^e_* S, R)$.  In particular, any map $\vartheta : F^e_* S \to R$ which factors through a map $F^e_* S(\Ram) \to R$ also factors through $\Tr : S(\Ram) \to R$.
\end{lemma}
\begin{proof}
We know
\begin{align*}
\Hom_R\bigl(F^e_* S(\Ram), R\bigr)  \cong  \Hom_R\Bigl(\bigl(F^e_* S(\Ram)\bigr) \otimes_S S, R \Bigr) \cong & \Hom_S\bigl(F^e_* S(\Ram), \Hom_R(S, R)\bigr) \\
\cong & \Hom_S\bigl(F^e_* S(\Ram), S(\Ram)\bigr).
\end{align*}
Given an element $\alpha \in \Hom_S\bigl(F^e_* S(\Ram), \Hom_R(S, R)\bigr)$ it is identified with a map $\beta \in \Hom_R\bigl(F^e_* S(\Ram), R\bigr)$ through $\Hom-\tensor$ adjointness by composing with the evaluation-at-1 map $\Hom_R(S,R) \to R$.
In our case,
\[
\vartheta \in \Hom_R\bigl(F^e_* S(\Ram), R)
\]
yields a map $\psi \in \Hom_S\bigl(F^e_* S(\Ram), S(\Ram)\bigr)$ such that, since $\Tr : S(\Ram) \to R$ corresponds to the evaluation-at-1 map $\Hom_R(S, R) \to R$, we know that $\vartheta = \Tr \circ \psi$.
\end{proof}

\begin{lemma}
\label{lem.TraceComposedWithRamMapsAndDeltas}
Suppose that $R \subseteq S$ is a finite separable extension of normal domains with $\pi : Y:=\Spec S  \to X:= \Spec R$ the corresponding morphism of schemes, and that $\Delta \geq 0$ is a $\bQ$-divisor on $X$ such that $\pi^* \Delta - \Ram$ is effective.  Then any map
\[
\vartheta \in \Hom_R\bigl(F^e_* S( \lfloor p^e \pi^* \Delta \rfloor + \Ram), R\bigr) \subseteq \Hom_R(F^e_* S, R)
\]
when restricted to $\vartheta :F^e_* S \to R$ factors through $\Tr : S \to R$.
\end{lemma}
\begin{proof}
Obviously $\Hom_R\bigl(F^e_* S(  \lfloor p^e \pi^* \Delta \rfloor +  \Ram), R\bigr) \subseteq \Hom_R\bigl(F^e_* S(\Ram), R\bigr)$ and so we can obtain a map
\[
\psi \in \Hom_S\bigl(F^e_* S(\Ram), S(\Ram)\bigr)
\]
with $\vartheta = \Tr \circ \psi$ by \autoref{lem.TraceComposedWithRamMaps}.

\begin{claim}
\label{clm.FactoringPsiContainment}
$\psi$ is contained in
\[
\Hom_S\bigl(F^e_* S(  \lfloor p^e \pi^* \Delta \rfloor  + \Ram), S(\Ram)\bigr) \subseteq \Hom_S\bigl(F^e_* S(\Ram), S(\Ram)\bigr).
\]
\end{claim}
\begin{proof}[Proof of claim]
Viewing $\vartheta$ as a section of $\Hom_R\bigl(F^e_* S(\Ram), R\bigr)$, we obtain a divisor $D_{\vartheta}$ with
\begin{align*}
 F^e_* S(D_{\vartheta}) \cong \Hom_R\bigl(F^e_* S(\Ram), R\bigr) \cong  \Hom_S\bigl(F^e_* S(\Ram), S(\Ram)\bigr) \cong & F^e_* S\bigl((1-p^e)(K_S - \Ram)\bigr) \\
= & F^e_* S\bigl((1-p^e)\pi^* K_R\bigr).
\end{align*}
and hence $D_{\vartheta} \sim (1-p^e) \pi^* K_R$.  Furthermore, by construction $D_{\vartheta} \geq \lfloor p^e \pi^* \Delta \rfloor$.  We have the composition
\[
\xymatrix{
F^e_* S(\Ram) \ar@/_2pc/[rr]_-{\vartheta} \ar[r]^-{\psi} & S(\Ram) \ar[r]^-{\Tr} & R.
}
\]
Since in this diagram, $\Tr : S(\Ram) \to R$ corresponds to the zero divisor, we see that $\psi$ corresponds to the same divisor as $\vartheta$ by \autoref{lem.CompositionOfMapsAndDivisors}.  In other words $\psi \in \Hom_S\bigl(F^e_* S(\Ram), S(\Ram)\bigr)$ yields the divisor $D_{\psi} \geq \lfloor p^e \pi^* \Delta \rfloor$.  This just means that
\[
\psi \in \Hom_S\bigl(F^e_* S(  \lfloor p^e \pi^* \Delta \rfloor  + \Ram), S(\Ram)\bigr).
\]
This proves the claim.
\end{proof}
By twisting and using the projection formula, we view $\psi$ as an element contained in $\Hom_S\bigl(F^e_* S( \lfloor p^e \pi^* \Delta \rfloor + (1-p^e) \Ram), S\bigr)$ and so by restriction, we obtain
\[
\psi : F^e_* S\bigl( \lfloor \pi^* \Delta + (p^e - 1) (\pi^* \Delta - \Ram)\rfloor \bigr) \to S.
\]
Further restricting to $F^e_* S$ (since $\pi^* \Delta - \Ram$ is effective) proves the lemma.
\end{proof}


\begin{theorem}
\label{thm.formula_signature_w/Deltas}
Suppose that $(R, \fram, k) \subseteq (S, \fran, \ell)$ is a finite separable extension of normal local domains with ramification / relative canonical divisor $\Ram$.  Let $\pi : Y=\Spec S  \to X = \Spec R$ be the corresponding morphism of schemes.  Suppose that $\Delta_X \geq 0$ is a $\bQ$-divisor on $X$ such that $\Delta_Y = \pi^* \Delta_X - \Ram \geq 0$. Let $f$ be the maximal number of direct $\Tr$-summands of $S$, then
\[
f \cdot s(S, \Delta_Y)= [L:K] \cdot s(R,\Delta_X).
\]
\end{theorem}
\begin{proof}
First note that if $\Tr(S) \subsetneq R$, then $f = 0$.  Furthermore, in this case, by \cite[Corollary 6.26]{SchwedeTuckerTestIdealFiniteMaps} $\tau(R, \Delta_X) = \Tr(\tau(S, \Delta_Y)) \subseteq \Tr(S) \subsetneq R$ and so $(R, \Delta_X)$ is not strongly $F$-regular and $s(R, \Delta_X) = 0$.  Thus the equality holds.  We may henceforth assume that $\Tr(S) = R$ so that $(S, \Delta_Y)$ is strongly $F$-regular if and only if $s(R, \Delta_X)$ is by \cite[Corollary 6.31]{SchwedeTuckerTestIdealFiniteMaps}.  Hence one side is zero if and only if the other is,  and so we may assume that $s(R, \Delta_X), s(S, \Delta_Y) > 0$.
Finally, again note that since $\Tr(S) = R$, $f = [\ell : k]$ by \autoref{lem.TrSurjectsImpliesTrSummandsIslk}.  We additionally remark that we can also assume without lost of generality that $\lfloor \Delta_X \rfloor =0$, otherwise the equality we plan to show would become trivially $0=0$.

Recall that by \cite[Lemma 3.6]{BlickleSchwedeTuckerFSigPairs1} one has that
\[
[\ell:k] = f= \lambda_R \bigr(S/\langle s\in S \mid \rho(s) \in \mathfrak{m}, \text{ for all } \rho \in \Tr \cdot S \rangle \bigl) = \lambda_R \bigr(\Tr \cdot S/\langle \rho \in \Tr\cdot S \mid \rho(S) \subseteq \mathfrak{m} \rangle \bigl).
\]

As in \autoref{subsec.FSig}, write
\[
\sG_e^{\Delta_X} = \Hom_R\bigl(F_*^e R(\lfloor p^e \Delta_X \rfloor ), R\bigr) \subseteq \Hom_R(F_*^e R, R) = \mathscr{C}_e^R
\]
and set $I_e^{\sG} = \{ r \in R \;|\; \varphi(F_*^e r) \subseteq \fram, \text{ for all } \varphi \in \sG_e^{\Delta_X}\}$.
This set $\sG_e^{\Delta_X}$ corresponds to the maps $\varphi \in \mathscr{C}_e^R$ which factors through a map $F_*^e R(\lfloor p^e \Delta_X \rfloor)\to R$ via the natural inclusion $R \subseteq R(\lfloor p^e \Delta_X \rfloor)$.
One defines and interprets $\mathscr{G}_e^{\Delta_Y}$ in the same fashion.

In order to recover an appropriate analog of \cite[Theorem 4.11]{TuckerFSigExists} we introduce
\[
\mathscr{D}_e:= \Hom_R\bigl(F_*^e S\bigl(\lfloor p^e \pi^*\Delta_X \rfloor + \Ram\bigr), R\bigr) \subseteq \Hom_R(F_*^e S, R).
\]
Let $J_e^{\sD}$ be the $R$-submodule of $S$
\[
J_e^{\sD}:= \bigl\{s\in S\mid \vartheta (F^e_* s) \in \mathfrak{m}, \text{ for all } \vartheta \in \mathscr{D}_e\bigr\}
\]
and write
\[
b_e:= \lambda_R\Bigl(F_*^e\bigl(S/J_e^{\sD}\bigr) \Bigr)=p^{e\alpha(R)} \lambda_R \bigl(S/J_e^{\sD}\bigr).
\]
The result now follows immediately from the two lemmas, \autoref{lem.FeStoRisGenRankTimesSRDelta} and \autoref{lem.FeStoRistraceSumandsTimessYDelta} below.
\end{proof}

\begin{lemma}
\label{lem.FeStoRisGenRankTimesSRDelta}
With notation as in the proof of \autoref{thm.formula_signature_w/Deltas},
\[
\lim_{e \rightarrow \infty} \frac{b_e}{p^{e(d+\alpha(R))}}=[L:K]\cdot s(R, \Delta_X).
\]
\end{lemma}
\begin{proof}
The proof of this is, \emph{mutatis mutandis}, the same as the proof of \cite[Theorem 4.11]{TuckerFSigExists}. Indeed, one has the following equalities,
\begin{align*}
[L:K] \cdot s(R, \Delta_X)= \rank_R(S) \cdot s(R, \Delta_X)
  &= \rank_R(S) \cdot \lim_{e \rightarrow \infty}{\frac{1}{p^{de}}e_{\text{HK}}\bigl(I_e^{\sG};R\bigr)}\\
  &= \lim_{e \rightarrow \infty}{\frac{1}{p^{de}}e_{\text{HK}}\bigl(I_e^{\sG};S\bigr)}\\
  &= \lim_{e \rightarrow \infty}{\frac{1}{p^{de}} \lambda_R\bigl(S/I_e^{\sG }S\bigr)}.
\end{align*}
The second equality holds in virtue of \cite[Corollary 3.7]{TuckerFSigExists} \cf \cite[Corollary 3.17]{BlickleSchwedeTuckerFSigPairs1} and \autoref{lem.FSigLimitWithGs}. The next equality follows for example from \cite[Lemma 1.3]{MonskyHKFunction}.  The last equality still follows from \cite[Corollary 3.7]{TuckerFSigExists}.

It only remains to verify that
\begin{equation}
\label{eq.EqualityOfLimitsJvsI}
\lim_{e \rightarrow \infty}{\frac{1}{p^{de}} \lambda_R\bigl(S/I_e^{\sG}S\bigr)} = \lim_{e \rightarrow \infty}{\frac{1}{p^{de}} \lambda_R\bigl(S/J_e^{\sD}\bigr)}.
\end{equation}
For this purpose we utilize \autoref{lem.CTrick}. First observe that
\[
R\bigl(\lfloor p^e\Delta_X \rfloor\bigr)\subseteq S\big(\pi^* \lfloor p^e \Delta_X \rfloor \big) \subseteq S\bigl(\lfloor p^e \pi^* \Delta_X + \Ram \rfloor\bigr).
\]
For any $r s \in I_e^{\sG} S$ and any $\vartheta \in \sD_e$, write $\vartheta'(F^e_* \blank) = \vartheta(F^e_* s \blank)$.  Then $\varphi' := \vartheta'|_{R\bigl(\lfloor p^e\Delta_X \rfloor\bigr)} \in \sG_e^{\Delta_X}$ so that $\vartheta(F^e_* rs) = \varphi'(F^e_* r) \in \fram$.  Hence
\begin{equation}
\label{eq.FirstContainment}
I_e^{\sG} S \subseteq J_e^{\sD}.
\end{equation}

Now let $0 \neq b \in R$ be such that $\Div_S(b) \geq \Ram$.  Then observe that for any $\bQ$-divisor $D$ on $\Spec R$, $\Div_S(b) \geq \Ram \geq \lfloor \pi^* D \rfloor - \pi^* \lfloor D \rfloor$.\footnote{To see this work locally with a separable extension of DVRs $R \subseteq S$ with uniformizers $r$ and $s$ such that $r = u s^n$.  Observe if $D = \lambda \Div_R(r)$, then
\[
\lfloor \pi^* D \rfloor \leq \pi^* D = \pi^* \{ \lambda \} \Div_R(r) + \pi^* \lfloor D \rfloor = n \{ \lambda \} \Div_S(s) + \pi^* \lfloor D \rfloor
\]
which implies that $\lfloor \pi^* D \rfloor \leq \lfloor n \{ \lambda \} \Div_S(s) + \pi^* \lfloor D \rfloor$ and so and $\lfloor n \{ \lambda \} \Div_S(s) \rfloor \leq \Ram$.
   }  It follows that
\begin{equation}
\label{eq.b2GetsRidOfRam}
b^2 S\bigl(\lfloor p^e \pi^* \Delta_X \rfloor + \Ram\bigr) \subseteq S\bigl(\pi^* \lfloor p^e \Delta_X \rfloor\bigr).
\end{equation}
Next choose a free module $G = R^{\oplus [L : K]} \subseteq S$ and a $0 \neq c \in R$ such that $c S \subseteq G$.  It follows that
\begin{equation}
\label{eq.cMultipliesDivisors}
cS\bigl(\lfloor p^e \pi^* \Delta_X \rfloor \bigr) = \bigl(c S \otimes_R R(\lfloor p^e \Delta_X \rfloor)\bigr)^{**}
\subseteq \bigl(G \otimes_R R(\lfloor p^e \Delta_X \rfloor)\bigr)^{**}
= \bigl(R(\lfloor p^e \Delta_X \rfloor)\bigr)^{\oplus [L : K]}
\end{equation}
where $\blank^{**}$ denotes reflexification as an $R$-module (or equivalently, since it can be viewed as S2-ification, as an $S$-module where appropriate).  Note the equalities and containments can be checked in codimension 1 where they are obvious.  Putting \autoref{eq.b2GetsRidOfRam} and \autoref{eq.cMultipliesDivisors} together yields:
\[
c b^2 S\bigl(\lfloor p^e \pi^* \Delta_X \rfloor + \Ram \bigr) \subseteq \bigl(R(\lfloor p^e \Delta_X \rfloor)\bigr)^{\oplus [L : K]}.
\]
Now choose $y \in J^{\sD}_e$ and
\[
\varphi \in \Hom_R\Bigl(F^e_* R\bigl(\lfloor p^e \Delta_X \rfloor \bigr)^{\oplus [L : K]}, R\Bigr) \subseteq \Hom_R(F^e_* G, R).
\]
Then, viewing things at the level of the field of fractions, $\varphi(F^e_* c b^2 \cdot \blank) \in \sD_e$ and so $\varphi(F^e_*c b^2 y) \in \fram$.  It follows that
\[
I^{\sG}_e S \subseteq J^{\sD}_e \subseteq \bigr(I^{\sG}_e S: c b^2\bigl)
\]
where the first containment was shown in \autoref{eq.FirstContainment}.  Now \autoref{eq.EqualityOfLimitsJvsI} follows from \autoref{lem.CTrick}.
\end{proof}

\begin{lemma}
\label{lem.FeStoRistraceSumandsTimessYDelta}
With notation as in the proof of \autoref{thm.formula_signature_w/Deltas},
\[
\lim_{e \rightarrow \infty} \frac{b_e}{p^{e(d+\alpha(R))}} = f \cdot s(S, \Delta_Y).
\]
\end{lemma}
\begin{proof} By applying the second part of \autoref{lem.FSigLimitWithGs}, we can compute $s(S, \Delta_Y)$ using
\[
\sG'_e={\sG'}_e^{\Delta_Y} = \Hom_R \bigl(F_*^eS(\lfloor p^e \Delta_Y\rfloor+\Ram), S\bigr),
\]
with the associated splitting numbers $a'_e = a_e^{\sG'} $.

We are going to prove that the following equality holds: $b_e=f\cdot a'_e$.  Taking limits will then prove the lemma.  We prove first that $b_e \geq f \cdot a'_e$. This can be verified by observing that the composition of an element $\psi \in \sG'_e$ with a map $\rho \in \Tr \cdot S$ gives an element $\vartheta:=\rho \circ \psi$ in $\sD_e$.  Indeed, $D_{\psi} \geq \lfloor p^e \Delta_Y \rfloor + \Ram$ and $D_{\rho} \geq \Ram$ and so by \autoref{lem.CompositionOfMapsAndDivisors}, we see that
\[
D_{\vartheta} \geq \bigl(\lfloor p^e \Delta_Y \rfloor + \Ram\bigr) + p^e \Ram = \Bigl(\lfloor p^e \pi^* \Delta_X - p^e \Ram \rfloor +p^e \Ram \Bigr) + \Ram  = \lfloor p^e \pi^* \Delta_X \rfloor + \Ram.
\]
It follows that $\vartheta \in \sD_e$ as claimed.  We can certainly construct $f \cdot a'_e$ distinct $R$-summands by such compositions and so $b_e \geq f \cdot a'_e$.


Now we prove that $b_e \leq f \cdot a'_e$.
Indeed, given any $\vartheta \in \sD_e$, \autoref{lem.TraceComposedWithRamMapsAndDeltas} provides a factorization $\vartheta=\Tr \circ \psi$ where $\psi$ belongs to $\Hom_S\bigl(F^e_* S(  \lfloor p^e \pi^* \Delta_X \rfloor  + \Ram), S(\Ram)\bigr)$ by \autoref{clm.FactoringPsiContainment}. However, as observed above $\lfloor p^e\Delta_Y \rfloor + (p^e+1) \Ram = \lfloor p^e \pi^* \Delta_X \rfloor + \Ram$, whence $\psi$ restricts by twisting by $S(\Ram)$ to a map $\psi: F_*^eS(\lfloor p^e\Delta_Y \rfloor + \Ram) \to S$, which is by definition an element of $\sG'_e$. Next note that \autoref{lem.Note_Trace} ensures that $\psi$ is surjective whenever $\vartheta$ is.  But now, if $N \subseteq F^e_* S$ is an $S$-module summand of $F^e_* S$ with no free $S$-summands whose projection maps are in $\sG_e'$, then $N$ has no free $R$-summands with corresponding projections in $\sD_e$.  The result follows.
\end{proof}

\begin{corollary}
\label{cor.BoundOnDegreeKeepDivisorEffective}
Suppose $(R, \fram)$ is a strictly Henselian $F$-finite normal local domain.  If $\Delta \geq 0$ is a $\bQ$-divisor on $\Spec R$ such that $(R, \Delta)$ is strongly $F$-regular,  then $\lfloor {1 / s(R, \Delta)}\rfloor$ is an upper bound on the maximal generic rank of a finite separable local extension $(R, \fram) \subseteq (S, \fran)$ so that $\pi^* \Delta - \Ram$ is effective where $\pi : \Spec S \to \Spec R$ is the induced map.
\end{corollary}
\begin{proof}
Suppose that $(R, \fram) \subseteq (S, \fran)$ is a finite separable local extension such that $\Delta_S = \pi^* \Delta - \Ram$ of generic rank $g > \lfloor {1 / s(R)}\rfloor$.  Hence $g > {1 / s(R, \Delta)}$. But $s(S, \Delta_S) = g s(R, \Delta) > 1$ and the result follows.
\end{proof}

The following should be viewed as a characteristic $p > 0$ analog of \cite[Proposition 1]{XuFinitenessOfFundGroups}.  First however, recall that if $X$ is projective variety over $k = \overline{k}$, $\Delta$ is a $\bQ$-divisor and $\sL$ is an ample line bundle, then we can form the the section ring $R = \bigoplus_{n \geq 0} H^0(X, \sL^n)$.  We have a canonical $k^*$-bundle map $\Spec R \setminus V(R_{>0}) \to X$.  We can then pull back $\Delta$ from $X$ and obtain a unique corresponding divisor $\Delta_R$ on $\Spec R$.  See \cite[Section 5]{SchwedeSmithLogFanoVsGloballyFRegular}.
\begin{corollary}
\label{cor.CoverOfGloballyFRegPair}
Suppose that $(X, \Delta_X)$ is a globally $F$-regular projective variety over an algebraically closed field $k$.  Suppose $\sL$ is ample on $X$ and that $R = \bigoplus_{n \geq 0} H^0(X, \sL^n)$ with $\Delta_R$ the corresponding divisor on $R$.  Then ${1 / s(R, \Delta_R)}$ is an upper bound on the generic rank of a finite separable cover $\pi : Y \to X$ with $Y$ normal such that $\Delta_Y = \pi^* \Delta_X - \Ram_{Y/X}$ is effective.  In particular, if $\Delta = 0$, then ${1 / s(R)}$ is an upper bound on a finite separable \etale-in-codimension-$1$ cover of any open set $U \subseteq X$ whose complement has codimension $\geq 2$ in $X$.
\end{corollary}
\begin{proof}
Write $S = \bigoplus_{n \geq 0} H^0(Y, \pi^* \sL^n) = \bigoplus_{n \geq 0} H^0\bigl(X, (\pi_* \O_Y) \otimes \sL^n\bigr)$.  We have a graded finite inclusion $R \subseteq S$ with associated $\mu : \Spec S \to \Spec R$.  Since $(X, \Delta_X)$ is globally $F$-regular, $(R, \Delta_R)$ is strongly $F$-regular by \cite[Proposition 5.3]{SchwedeSmithLogFanoVsGloballyFRegular}.  Next observe that the generic rank of $\pi_* \O_Y$ over $\O_X$ is the same as the generic rank of $S$ over $R$ since if $\O_X^{\oplus m} \subseteq \pi_* \O_Y$ has torsion cokernel, so does the corresponding inclusion $R^{\oplus m} \subseteq S$.  Finally, observe that $\Ram_{S/R}$ simply corresponds to $\Ram_{Y/X}$ and so $\Delta_S = \mu^* \Delta_R - \Ram_{S/R}$ corresponds to $\Delta_Y$ and $\Delta_S$ is effective if and only if $\Delta_Y$ is.
The corollary now follows immediately by completion (or Henselization) and \autoref{cor.BoundOnDegreeKeepDivisorEffective}.
\end{proof}

\section{Finiteness of the fundamental group of a strongly $F$-regular singularity}
\label{FinitenessFundGroup}
In this section we prove our main result, namely; finiteness of the \'{e}tale fundamental group of a strongly $F$-regular singularity.
  Let $(R,\mathfrak{m}, k)$ be a normal $F$-finite and strongly $F$-regular strictly local domain of prime characteristic $p>0$. 
We will demonstrate finiteness of the \etale{} fundamental group of $U' \subseteq U:=\pSpec (R)$ where $U'$ is the complement of a closed subset $Z$ of $X:=\Spec(R)$ through $\fram$ of codimension at least $2$. In particular, since strongly $F$-regular rings are normal, the singular locus has codimension at least 2, and we can take $U'=U_{\text{reg}}$ as the regular locus. 

\begin{theorem}
\label{thm.FinitenessFundGroup}
Let $(R,\mathfrak{m}, k)$ be a normal $F$-finite and strongly $F$-regular strictly local domain of prime characteristic $p>0$, with dimension $d\geq 2$. Let $Z\ni \fram$ be a closed subscheme of $X:=\Spec(R)$ of codimension at least $2$ with complement $U$.
Then the \'{e}tale fundamental group of $U$, \ie $\pi_1:=\fg\bigl(U,\bar{x}\bigr)$, is finite. Furthermore, the order of $\pi_1$ is at most $1/s(R)$ and is prime to $p$. For example, $Z=\{\fram\}$ and $U=\Spec^{\circ}(R)$.
\end{theorem}
\begin{proof}
First note that since $R$ is a normal domain, $U = X \setminus Z$ is always connected.
Next, notice that since the trace is surjective, all the coverings we consider are (cohomologically) tamely ramified, in particular; given we have taken the hypothesis $k=k^{\text{sep}}$, all the residue field extensions in this category are trivial, see \autoref{lem.SepResFields}.

In this way, to show that $\pi_1$ is finite one just has to show that a sequence of module-finite local inclusions
$$(R, \mathfrak{m}, k) \subseteq (S_1, \mathfrak{n}_1, k) \subseteq (S_2, \mathfrak{n}_2, k) \subseteq (S_3, \mathfrak{n}_3, k) \subseteq \cdots $$
in which the consecutive inclusions $(S_i, \mathfrak{n}_i, k) \subseteq (S_{i+1}, \mathfrak{n}_{i+1}, k)$ are \'{e}tale in codimension $1$, and the extensions $(R, \mathfrak{m}, k) \subseteq (S_i, \mathfrak{n}_i, k)$ are all Galois, stabilizes; by this we mean that $S_i = S_{i+1}$ for all $i\gg 1$. In fact, what we have is that all but finitely many consecutive inclusions are \'{e}tale everywhere, this is a direct consequence of the \autoref{cor.F_Signature_goes_up_under_ramification}: because if this is not the case, then the $F$-signature of the local rings $S_i$ will eventually be arbitrarily large as $i$ grows (since we started with $s(R) > 0$), but the $F$-signature of any ring is well-known to be at most $1$. Hence, eventually one has equalities since in this setting \'{e}tale-ness guarantees equality (the extension will be free because it is \etale, but at most $1=[\ell : k]$ free summand is allowed or possible because the extensions of residue fields are all trivial).

For the upper bound on the order of $\pi_1$, just noticed that in this case there will be a Galois extension $(R,\fram) \subseteq (S^\star, \fran^\star)$ representing $\sG$. This is going to dominate any other such extension. To get such $(S^\star, \fran^{\star})$, let $(S, \fran)$ be a maximal element in a chain as above (this corresponds to a maximal element a chain of \etale{} extensions over $U$).  If it does not dominate some other extension $(R, \fram) \subseteq (T, \mathfrak{o}) \subseteq K^{\textnormal{sep}}$ which is also \etale{} over $U$, then both $(S, \fram)$ and $(T, \mathfrak{o})$ can be dominated by a larger extension also \etale{} over $U$.  Thus we may take $S^{\star} = S$.

We now have $\pi_1 = \Aut_{\sG} S^{\star}=\text{Gal}(L^{\star}/ K)$, the second equality holds because of the Galois condition, thereby
\[
\# \pi_1 = \# \Aut_{\sG} S^{\star} = [L^{\star}:K]\leq \frac{1}{s(R)},
\]
the inequality at the end is just given by \autoref{cor.BoundGenericRnk}. By \autoref{cor.PurityOfBranchLocusForCovers}, it also follows the order of the group is prime to $p$.
\end{proof}
\begin{remark} We remark we need to assume our local ring expressing the singularity is strictly Henselian and not just Henselian since otherwise the associated fundamental group would contain $\text{Gal}(k^{\text{sep}}/k)$, which might easily be infinite, for instance; for perfect fields the separable closure coincides with the algebraic closure.  For example, this is infinite for $k= \mathbb{F}_p$. Nonetheless, under the hypothesis $\#\text{Gal}(k^{\text{sep}}/k)< \infty$ the same result would follow.
\end{remark}
\begin{remark}
As we have observed several times before, the surjectivity of the trace imposes strong tameness on the ramification, namely cohomological tameness. So that the \etale{} fundamental group we dealt with in \autoref{thm.FinitenessFundGroup} is actually the same as the/any tame fundamental group. In fact, as we noticed, the order of the group is prime to $p$.
\end{remark}

From the proof of \autoref{thm.FinitenessFundGroup} we get the following statement, which is a local and positive characteristic analog for \cite[Theorem 1.1]{GrebKebekusPeternellEtaleFundamental}.

\begin{scholium}
Let $(R, \mathfrak{m}, k)$ be a strongly $F$-regular local domain. In any chain
$$(R, \mathfrak{m}) \subseteq (S_1, \mathfrak{n}_1) \subseteq (S_2, \mathfrak{n}_2) \subseteq (S_3, \mathfrak{n}_3) \subseteq \cdots$$
of module-finite local \etInCdOne{} inclusions of normal local domains all but finitely many of the extensions are \etale{} everywhere.
\end{scholium}

The following corollary is the local positive characteristic analog of \cite[Theorem 1.5]{GrebKebekusPeternellEtaleFundamental}. However, in our local case it follows from a general result for singularities with finite fundamental group. So that it is not that interesting as in the global case.

\begin{corollary}
\label{cor.MaximalCover3}
Suppose that $(R, \fram)$ is a strictly local $F$-finite strongly $F$-regular domain.  Then inside any fixed separable closure of its fraction field $K$, there exists a unique largest finite local \etInCdOne{} extension $(R, \fram) \subseteq (S, \fran)$ of normal domains so that $(S, \fran)$ has no non-trivial finite local \etInCdOne{} extension $(S, \fran) \subsetneq (T, \mathfrak{o})$, i.e. with trivial fundamental group.
\end{corollary}

Compare the following with \cite[Theorem 1.10]{GrebKebekusPeternellEtaleFundamental}.

\begin{corollary}
Suppose $(R, \fram) \subseteq (S, \fran)$ as in \autoref{cor.MaximalCover3} with $\pi : \Spec S \to \Spec R$ the induced morphism.  If $D$ is a $\bQ$-Cartier Weil divisor on $\Spec R$ such that $n$, the index of $D$, is not divisible by $p$, then $n \leq {1 / s(R)}$ and $n \;|\; [K(S) : K(R)]$.  Furthermore, $\pi^* D$ is Cartier on $\Spec S$.
\end{corollary}
\begin{proof}
If $D$ is $\bQ$-Cartier with index $n$, not divisible by $p$, then the cyclic cover $R \subseteq R \oplus R(-D) \oplus \dots \oplus R(-(n-1)D) = T$ is \etInCdOne{} by \cite[Example 2.8]{WatanabeFRegularFPure}.  Now, $T$ is a domain by \cite[Corollary 1.9]{TomariWatanabeNormalZrGradedRings} hence local since $R$ is Henselian.  The extension $R \subseteq T$ has generic rank $n$ and the first statements follows.  Note that $\pi^* D$ is still $\bQ$-Cartier but it must be Cartier since otherwise we could take a cyclic cover on $S$ which has no non-trivial \etInCdOne{} covers.
\end{proof}

By taking cones we also have the following.

\begin{corollary}
\label{cor.FundGroupOfGFR}
Suppose that $X$ is a projective globally $F$-regular variety over an algebraically closed field and that $Z \subseteq X$ is a closed subset of codimension $\geq 2$.  Let $U = X \setminus Z$.  Then $\pi_1:=\fg\bigl(U)$ is finite of order prime to $p$.
\end{corollary}
\begin{proof}
We can bound the degree of any finite \etale{} cover of $U$ as before in \autoref{cor.CoverOfGloballyFRegPair} and prove that its order is relatively prime to $p$.  The result then follows exactly as in \autoref{thm.FinitenessFundGroup}.
\end{proof}

\begin{remark}
In the case that $\dim X \leq 3$ and $p \geq 11$, if $X$ is globally $F$-regular it is also rationally chain connected by \cite{GongyoLiPatakfalviSchwedeTanakaZong}.  Hence by \cite{Chambert-LoirPointsRationnels}, if it is also smooth, the \etale{} fundamental group is finite of order prime to $p$ (see also \cite[Theorem 4.13]{KollarShafarevich} and \cite[Section 11.2]{GrebKebekusPeternellEtaleFundamental}).  On the other hand, \autoref{cor.FundGroupOfGFR} can also be thought of as evidence that (smooth) globally $F$-regular varieties are rationally chain connected in all dimensions.
\end{remark}

It is natural to ask whether the characteristic $p > 0$ results of this paper imply the characteristic zero results of \cite{XuFinitenessOfFundGroups} and some of the results of \cite{GrebKebekusPeternellEtaleFundamental}.  Unfortunately, we do not know how to reduce local algebraic fundamental groups to characteristic $p \gg 0$ by spreading out.  However, if we had a positive answer to the following question, a number of characteristic zero results would immediately follow.

\begin{question}
Suppose that $(R, \fram)$ is the local ring of a singularity in characteristic zero.  Consider a family of characteristic $p > 0$ reductions $(R_p, \fram_p)$.  Is\footnote{We write this imprecisely to make the idea clear.  If the way we wrote it bothers the reader, simply assume everything is defined over $\bZ$.  Otherwise the usual reduction to characteristic $p > 0$ accoutrements must be employed.}
\[
\limsup_{p \rightarrow \infty} s(R_p) > 0?
\]
\end{question}
Examples seem to suggest that this is the case.
Of course, this is only a special case of the following question which a number of people have already considered.
\begin{question}
Does $ \lim_{p \rightarrow \infty} s(R_p)$ have a geometric interpretation, or at least some geometric lower bounds?
\end{question}
Indeed, answers to this question would yield effective versions of many of the results of \cite{XuFinitenessOfFundGroups} and \cite{GrebKebekusPeternellEtaleFundamental}.




\bibliographystyle{skalpha}
\bibliography{MainBib}
\end{document}